\documentclass{amsart}

\usepackage{latexsym,amssymb,enumerate, amsmath, xcolor}

\newenvironment{enumeratei}{\begin{enumerate}[\upshape (a)]}
    {\end{enumerate}}

\begin{document}

\newtheorem{theorem}{Theorem}[section]

\newtheorem{lemma}[theorem]{Lemma}

\newtheorem{proposition}[theorem]{Proposition}

\newtheorem{corollary}[theorem]{Corollary}
\newtheorem*{thmA}{Theorem A}
\newtheorem*{thmB}{Theorem B}
\newtheorem*{corC}{Corollary}
\newtheorem*{proD}{Proposition} 

\theoremstyle{definition}
\newtheorem{definition}[theorem]{Definition}

\newtheorem{question}[theorem]{Question}

\newtheorem{conjecture}[theorem]{Conjecture}

\newcommand{\F}{\ensuremath{\mathbb F}}
\newcommand{\Ef}{\ensuremath{\mathbb E}}
\newcommand{\K}{\mathbb K}
\newcommand{\N}{\mathbb N}

\newcommand{\R}{\mathcal R}
\def\irr#1{{\rm Irr}(#1)}
\def\ibr#1#2{{\rm IBr}_#1(#2)}
\def\cs#1{{\rm cs}(#1)}
\def\cd#1{{\rm cd}(#1)}
\def\m#1{{\rm m}(#1)}
\def\n#1{{\rm n}(#1)}
\def\cent#1#2{{\bf C}_{#1}(#2)}
\def\hall#1#2{{\rm Hall}_#1(#2)}
\def\syl#1#2{{\rm Syl}_#1(#2)}
\def\nor{\trianglelefteq\,}
\def\norm#1#2{{\bf N}_{#1}(#2)}
\def\NA#1#2{{\Lambda}_{#1}(#2)}
\def\oh#1#2{{\bf O}_{#1}(#2)}
\def\zent#1{{\bf Z}(#1)}
\def\hyperzent#1{{\bf {Z_{\infty}}}(#1)}
\def\sbs{\subseteq}
\def\gen#1{\langle#1\rangle}
\def\aut#1{{\rm Aut}(#1)}
\def\alt#1{{\rm Alt}(#1)}
\def\sym#1{{\rm Sym}(#1)}
\def\out#1{{\rm Out}(#1)}
\def\gv#1{{\rm Van}(#1)}
\def\fit#1{{\bf F}(#1)}
\def\fitt#1{{\bf F}_2(#1)}
\def\diam#1{{\rm diam}(#1)}
\def\frat#1{{\bf \Phi}(#1)}
\def\GF#1{{\rm GF}(#1)}
\def\SL#1{{\rm SL}(#1)}
\def\V#1{{\rm V}(#1)}
\def\Vv#1{{\rm V_v}(#1)}
\def\E#1{{\rm E}(#1)}
\def\Ev#1{{\rm E_v}(#1)}
\def\gammava#1{{\Gamma_{\rm v}}(#1)}
\def\nfr#1{\widehat{#1}^{*}}

\title{On Vanishing Class Sizes in Finite Groups}

\author[M. Bianchi et al.]{Mariagrazia Bianchi}
\address{Mariagrazia Bianchi, Dipartimento di Matematica F. Enriques,\newline
Universit\`a degli Studi di Milano, via Saldini 50,
20133 Milano, Italy.}
\email{mariagrazia.bianchi@unimi.it}

\author[]{Julian M.A. Brough}
\address{Julian M.A. Brough, Technische Universit\"at Kaiserslautern,\newline 67663, Kaiserslautern, Germany.}
\email{brough@mathematik.uni-kl.de}

\author[]{Rachel D. Camina}
\address{Rachel D. Camina, Fitzwilliam College, Cambridge, CB3 0DG,
UK.} \email{rdc26@dpmms.cam.ac.uk}

\author[]{Emanuele Pacifici}
\address{Emanuele Pacifici, Dipartimento di Matematica F. Enriques,
\newline Universit\`a degli Studi di Milano, via Saldini 50,
20133 Milano, Italy.}
\email{emanuele.pacifici@unimi.it}

\thanks{The first and the fourth author are partially supported by the Italian INdAM-GNSAGA, and by the PRIN 2015TW9LSR\_006 ``Group Theory and Applications".}

\begin{abstract}
Let \(G\) be a finite group. An element \(g\) of \(G\) is called a \emph{vanishing element} if there exists an irreducible character \(\chi\) of \(G\) such that \(\chi(g)=0\); in this case, we say that the conjugacy class of \(g\) is a vanishing conjugacy class. In this paper, we discuss some arithmetical properties concerning the sizes of the vanishing conjugacy classes in a finite group. 
\end{abstract}

\maketitle

\section{Introduction}

Many authors have investigated the relationship between the structure of a finite group $G$ and arithmetical data connected to $G$.
The arithmetical data can take various forms: for example, authors have considered the set of conjugacy class sizes, or the set of character degrees. The link between these different sets is also of interest, as demonstrated by the following result by
C. Casolo and S. Dolfi. Suppose $p$ and $q$ are distinct primes and $pq$ divides the degree of some irreducible complex character of \(G\);
then $pq$ also divides the size of some conjugacy class of \(G\) \cite[Theorem A]{casolo}.
As an important step in the proof of this result, the authors consider groups for which $p$ and $q$ both divide a conjugacy class size but $pq$ does not, and they show that such groups are $\{p,q\}$-solvable \cite[Theorem B(i)]{casolo}.

Recently, instead of considering all conjugacy class sizes, 
authors have been considering a subset of conjugacy class sizes ``filtered" by the irreducible characters, namely, the set of 
{\it vanishing conjugacy class sizes} (see \cite{brough1}, \cite{brough2}, \cite{dolfi} and also \cite {spiga} for related
properties of vanishing elements). An element
$g \in G$ is called a vanishing element if
there exists an irreducible character $\chi$ of $G$ such that $\chi(g) =0$, and the conjugacy class of such an element is called a vanishing conjugacy class of $G$. 
Motivated by Casolo and Dolfi's results, we investigate some arithmetical properties of the set of vanishing conjugacy class sizes.

This context is neatly portrayed by the \emph{prime graph} of $G$ for class sizes. Recall that, given a finite nonempty set of positive integers \(X\), the prime graph on $X$ has vertex set defined as the set of all prime numbers that are divisors of some element in \(X\), and edge set consisting of pairs $\{p,q\}$ such that $pq$ divides some element of $X$. 
When \(X=\{x^G\;:\;x\in G\}\) is the set of conjugacy class sizes of a finite group \(G\), we denote by  \(\Gamma(G)\) the prime graph on \(X\), and by \(\V G\) the set of vertices of \(\Gamma(G)\). Also, we write \(\gammava G\) and \(\Vv G\) for the corresponding objects in the case when \(X\) is the set of vanishing conjugacy class sizes of \(G\). 

In \cite{dolfi} the authors investigate when, for a finite group \(G\), a prime $p$ is not an element of $\Vv  G$; they prove that such a group \(G\) is \(p\)-nilpotent, with abelian Sylow $p$-subgroups. Note that \(\Vv  G\) can be strictly smaller than \(\V G\), as shown for instance by the symmetric group on three objects. This can actually occur also for nonsolvable groups (see \cite[Example 4.1]{dolfi}). However, our first result gives a condition to ensure this does not happen.

\begin{proD} 
\label{samevertices}
Let $G$ be a finite group, and suppose $G$ has a nonabelian minimal normal subgroup. Then $\V G = \Vv  G$.
\end{proD}

Thus, for the two vertex
sets to be the same in a nonsolvable group, it seems important ``where" the nonsolvability of the group lies.

Still in the spirit of the work by Casolo and Dolfi, we carry out an investigation concerning the edges of the vanishing graph. In particular, 
is the ``vanishing version" of \cite[Theorem B(i)]{casolo} true? That is, if the edge $\{p,q\}$ is missing in the vanishing graph, 
but both $p$ and $q$ are vertices, is the group
$\{p,q \}$-solvable? 

As our main result shows, the answer is affirmative under the same assumptions as in the above Proposition. 

\begin{thmA} Let \(G\) be a finite group, and suppose \(G\) has a nonabelian minimal normal subgroup. If \(p\) and \(q\) are in \(\V G\), but there is no vanishing conjugacy class of \(G\) whose size is divisible by \(pq\), then \(G\) is \(\{p,q\}\)-solvable.
\end{thmA}

As shown by \cite[Example 4.1]{dolfi}, Theorem~A fails in general if \(G\) does not have a nonabelian minimal normal subgroup.
An important step in the proof of the above result is the following Theorem B, that is the vanishing version of \cite[Theorem 9]{casolo}. 

\begin{thmB} Let $G$ be a finite group with trivial Fitting subgroup. Then every prime divisor of \(|G|\) is in \(\Vv G\), and \(\gammava G\) is a complete graph.
 \end{thmB}

We point out that another key ingredient in the proof of Theorem A is Corollary~\ref{coro}, that turns out to be a useful tool in locating vanishing elements, and may be of interest in its own right.

As an immediate consequence of Theorem~A, we get the following \(p\)-solvability criterion. Recall that a vertex of a graph is called \emph{complete} if it is adjacent to all the other vertices.

\begin{corC}
Let \(G\) be a finite group and \(p\) a prime. Suppose \(G\) has a nonabelian mi\-nimal normal subgroup. If \(p\) is not a complete vertex of \(\gammava G\), then \(G\) is \mbox{\(p\)-solvable.}
\end{corC}

Throughout this paper, every group is assumed to be a finite group.
  
\section{Preliminaries}

In this section we gather together previously known results that will be of use. We denote the set of all vanishing elements of the group $G$ by ${\rm Van}(G)$, whereas, as customary, \(\pi(G)\) denotes the set of prime divisors of \(|G|\).

\begin{lemma}\label{lemma1} Let $N$ be a normal subgroup of $G$. Then $\gammava  {G/N}\) is a subgraph of \(\gammava  G\).
\end{lemma}
\begin{proof} We combine two results. Firstly note that $|(xN)^{G/N}|$ divides $|x^G|$. Secondly, by \cite[Lemma 2.1]{spiga}, 
if $xN \in {\rm Van}(G/N)$ then $xN \subseteq {\rm Van}(G)$.
\end{proof}

\begin{lemma}\cite[Proposition 2.5]{spiga}\label{lemma3} Let $M$ and $N$ be normal subgroups of a group $G$ with $M \cap N =1$. Suppose
 there exists $m \in M \cap {\rm Van}(G)$. Then $mn \in {\rm Van}(G)$ for all $n \in N$.
\end{lemma}

\begin{proposition}\label{three} \cite[Lemma 8]{casolo} Let $G$ be a permutation group on a finite set $\Omega$ and $p,q$ distinct primes. Then there
 exist two nonempty subsets $\Gamma_1, \Gamma_2$ of $\Omega$ with $\Gamma_1 \cap \Gamma_2 = \emptyset$ and $\{p,q\} \cap \pi(G) \subseteq 
 \pi(|G:G_{\Gamma_1} \cap G_{\Gamma_2}|)$.
\end{proposition}

(Here, for \(i\in\{1,2\}\), \(G_{\Gamma_i}\) denotes the setwise stabilizer of \(\Gamma_i\) in \(G\).) 

Recall, an irreducible character $\chi$ of $G$ is said to have $q$-defect zero for some prime $q$, if $q$ does not divide $|G|/ \chi(1)$.
If $\chi$ is such a character and $g$ is an element of $G$ with order divisible by $q$ then $\chi(g)=0$, i.e. $g$ is a vanishing element
\cite[Theorem 8.17]{isaacs}. Thus the following is useful.

\begin{theorem}\label{defect}\cite[Corollary 2]{ono} Let $G$ be a nonabelian simple group and $q$ a prime divisor of $|G|$. Then $G$ has
 an irreducible character of $q$-defect zero unless one of the following holds.
 \begin{enumeratei}
 \item The prime $q$ is $2$ and $S$ is isomorphic to either ${\rm M}_{12}$, ${\rm M}_{22}$, ${\rm M}_{24}$, ${\rm J}_2$, ${\rm HS},$ ${\rm Suz}$, ${\rm Ru}$, ${\rm Co}_1$,
 ${\rm Co}_3$, ${\rm BM}$ or $\alt n$ for various values of \(n\geq 7\)
 \item The prime $q$ is $3$ and $S$ is isomorphic to either ${\rm Suz}, {\rm Co}_3$ or $\alt n$ for various values of \(n\geq 7\).
 \end{enumeratei}
\end{theorem}

The above result will often be used in conjunction with the following.

\begin{lemma}\label{defect2}\cite[Lemma 2.7]{spiga} Let $G$ be a group, $N$ a normal subgroup of $G$ and $q$ a prime divisor of
 $|N|$. If $N$ has an irreducible character of $q$-defect zero, then every element of $N$ of order divisible by $q$ is a vanishing element
 of $G$.
\end{lemma}

As an immediate consequence of the two previous statements, we get the following result.

\begin{proposition}\cite[Corollary 2.9]{spiga}\label{prop1} Let $M$ be a nonabelian minimal normal subgroup of a group $G$ and suppose $p$ is a prime 
 divisor of $|M|$. If $p \geq 5$ then every element of $M$ with order divisible by $p$ is a vanishing element of $G$.
\end{proposition}

Finally, we will freely use without references some basic facts of Character Theory such as Clifford Correspondence, properties of coprime actions, and
elementary properties of conjugacy class sizes; for instance, recall that a prime \(p\) does not divide the size of any conjugacy class of a group \(G\) if and only if \(G\) has a central Sylow \(p\)-subgroup.

\section{Theorem B}

We start with a simple lemma.

\begin{lemma}\label{products} Let $G$ be a group, $M$ a normal subgroup of $G$ with trivial centre and $C=\cent GM$. Then for any $g \in C$ and
 $h \in M$ the conjugacy class size $|(gh)^G|$ is divisible by both $|g^C|$ and $|h^M|$. Furthermore if $h$ or $g \in {\rm Van}(G)$ then
 $gh \in {\rm Van}(G)$.
\end{lemma}

\begin{proof} As $M$ has trivial centre $MC \cong M \times C$. Thus $|(gh)^{MC}| = |g^C||h^M|$. As $MC$ is normal in $G$ it follows
that $|(gh)^{MC}|$ divides $|(gh)^G|$. Finally, if $h$ or $g \in {\rm Van}(G)$ apply
Lemma \ref{lemma3}.
\end{proof}

The following lemma helps us to identify vanishing elements.

\begin{lemma}\label{lemma2} Suppose $G$ has a unique minimal normal subgroup $M$ which is nonabelian. Then for all $p \in \pi(G)$ there exists
 $g \in {\rm Van}(G)$ such that $p$ divides $|g^G|$. Thus $\pi(G) = \Vv  G$. Furthermore, $g$ can be chosen to lie in $M$.
\end{lemma}

\begin{proof} As $M$ is nonabelian, we have $M = S_1 \times \cdots \times S_n$ where every $S_i$ is isomorphic to a nonabelian simple
group $S$. We will denote by $N$ the kernel of the action
of $G$ by conjugation on $\{S_1, \ldots, S_n\}$. 

Let \(p\) be a prime divisor of \(|G|\); our aim is to find an element \(g\in M\cap{\rm Van}(G)\) such that \(p\) divides \(|g^G|\). We will treat separately the cases \(p\mid |M|\) and \(p\nmid |M|\). 

Let us start from the latter case. In this situation, \(p\) divides \(|G/M|\), and we first suppose that $p$ actually divides $|G/N|$. Using
Proposition~\ref{three}, we can choose two nonempty susbets $\Gamma_1, \Gamma_2$ of $\Omega = \{S_1, \ldots, S_n \}$, such that $\Gamma_1 \cap \Gamma_2 = \emptyset$ and
$p \in \pi(|G:G_{\Gamma_1} \cap G_{\Gamma_2}|)$. Certainly we can find nontrivial elements $u$ and $v$ in $S$ of different orders and such that the order of $v$ is divisible by some prime $r$ greater
than 3. For $S_{\alpha} \in \Gamma_1$ and $S_{\beta} \in \Gamma_2$
let $u_{\alpha}$ and $v_{\beta} $ correspond, respectively, to \(u\) and \(v\) (via the isomorphisms \(S_{i}\simeq S\)), and set $g$ to be the element in $M$ given by
$$g = \Pi u_{\alpha}\Pi v_{\beta}.$$ As $r$ divides the order of $v$, it also divides the order of $g$ and hence $g$ is vanishing
in $G$ by Proposition \ref{prop1}.
Let $x$ be an element in $\cent G g$, and consider $S_\alpha \in \Gamma_1$. Let $x$ act on $M$ by conjugation; since $M$ is nonabelian, the factors of $g$ are permuted. As the orders of $u$ and $v$ are different, it follows that $u_{\alpha}^x = u_{\alpha '}$ for some
$\alpha'$ with $S_{\alpha '} \in \Gamma_1$. Moreover it follows that $S_{\alpha}^x = S_{\alpha'} \in \Gamma_1$ and so
$x \in G_{\Gamma_1}$. Similarly we have $x \in G_{\Gamma_2}$, whence $\cent G g \leq G_{\Gamma_1} \cap G_{\Gamma_2}$. As a consequence,
$|G:G_{\Gamma_1} \cap G_{\Gamma_2}|$ divides $|g^G|$ and, in particular, $p$ divides $|g^G|$.

Suppose now that $p$ does not divide $|G/N|$, so $p$ divides $|N|$. Take \(i\) in \(\{1,...,n\}\), and adopt the bar convention for the factor group \(N/\cent N{S_i}\); then \(\overline{N}\) is an almost simple group with socle \(\overline{S_i}\simeq S\). Moreover, since \(\bigcap \cent N{S_i}=\cent N M=1\), the subgroup \(N\) can be embedded in the direct product of the factor groups \(N/\cent N{S_i}\) and, as these factor groups all have the same order (because \(G\) transitively permutes the \(S_i\)), it is easy to see that \(p\) divides \(|\overline{N}|\). However, $p$ does not divide $|M|$, and thus it does not divide $|S_i|$; 
as a consequence, $S$ is a simple group of Lie type (see for instance \cite{gorenstein}).
Now, by \cite[Lemma 6(a)]{casolo}, there exists an element $\overline{g} \in \overline{S_i}$ such that $p$ divides $|\overline{g}^{\overline{N}}|$, which is in turn a divisor of \(|g^{G}|\) (here \(g\) can be chosen in \(S_i\), hence in \(M\)). As $g$ is a vanishing element of \(G\) by Theorem \ref{defect} and Lemma~\ref{defect2}, we are done in this case as well.

Finally, let us assume \(p\mid |M|\). If \(M\) has an irreducible character of \(q\)-defect zero for every prime divisor \(q\) of \(|M|\), then every nontrivial element of \(M\) lies in \({\rm{Van}}(G)\) by Lemma~\ref{defect2}. Now, just take an element \(g\in S_1\) such that \(p\mid |g^{S_1}|=|g^M|\), and we are done. On the other hand, if there exists a prime \(q\in\pi(M)\) such that \(M\) does not have an irreducible character of \(q\)-defect zero (so the same holds for \(S_1\)), then we apply Lemma 2.2 in \cite{dolfi}: there exists an element $g$ of $S_1$ whose conjugacy class in \(M\) has size divisible by all primes in $\pi(M)$, and there exists an
irreducible character $\theta$ of $S_1$ such that $\theta(g) =0$ and $\theta$ extends irreducibly to ${\rm Aut}(S_1)$. Moreover, Lemma~5 of \cite{bianchi} yields that $\theta\times\theta\times\cdots\times\theta\in{\rm{Irr}}(M)$ extends irreducibly to $G$, and thus $g$ is vanishing in~$G$.  
\end{proof}

We are now ready to prove the proposition mentioned in the Introduction, that we state again.

\begin{proD} 
Let $G$ be a group, and suppose $G$ has a nonabelian minimal normal subgroup. Then $\V G = \Vv  G$.
\end{proD} 

\begin{proof} Let $M$ be a nonabelian minimal normal subgroup of $G$, and set $C=\cent GM$. 
Then $\overline{G} = G/C$ has a unique minimal
normal subgroup, which is isomorphic to $M$. Suppose $p \in \V G$. If $p \in \pi(\overline{G})$ then $p \in \Vv  {\overline{G}}$ by Lemma \ref{lemma2}, and hence 
 $p \in \Vv  G$ by Lemma~\ref{lemma1}.
Thus we can assume that $p$ does not divide $|\overline{G}|$. If $C$ does not have a central Sylow \(p\)-subgroup, then there exists $g \in C$ with $p$ dividing $|g^C|$. As $M$ is nonabelian, there exists $h \in M\cap{\rm{Van}}(G)$ by Proposition \ref{prop1}. Now apply Lemma \ref{products} to get that $gh$ is vanishing in $G$ with conjugacy class size
divisible by $p$, thus $p \in \Vv  G$ as required. Finally, suppose
that $C$ has a central Sylow $p$-subgroup \(P\). Then $P$, which is a Sylow $p$-subgroup of \(G\) as well, is (abelian and) normal in $G$. If $p \not\in \Vv  G$ then, by \cite[Theorem A]{dolfi}, $G$ has a normal $p$-complement, thus $P$ is central in \(G\). 
But this would contradict $p \in \V G$, and the proof is complete.
\end{proof}

Thus, when $\fit  G=1$ we have established that $\Vv  G=\pi(G) $. We now turn our attention to edges in the vanishing graph and, after the following proposition, we will prove Theorem B.

\begin{proposition}\label{almost} Let $G$ be an almost simple group with socle \(S\), and  let $p$, $q$ be distinct
 primes in $\pi(G)$. Then $p$ and \(q\) are adjacent vertices of \(\gammava G\). Moreover, there exists an element $g \in S$ such that $pq$ divides $|g^G|$ and
 $g$ is vanishing in $G$.
\end{proposition}

\begin{proof} If $p,q \in \pi(S)$, then \(p\) and \(q\) are adjacent vertices of \(\Gamma(S)\) by \cite[Theorem 9]{casolo}. So there exists $g \in S$ with $pq$ dividing
$|g^S|$ and thus $|g^G|$. If $S$ has an irreducible character of $q$-defect zero for all primes $q$, then $g$ is vanishing in \(G\) by 
Lemma \ref{defect2} and we are done. Thus, we can assume $S$ does not have an irreducible character of $q$-defect zero for some prime $q$, and the same argument as in the last paragraph of Lemma~\ref{lemma2} yields the conclusion.

If $p$ or $q$ does not divide $|S|$, then $S$ is a simple group of Lie type. We note that in the proof of \cite[Proposition 7]{casolo} the
authors produce, in each case, an element of $S$ with conjugacy class size divisible by $p$ and $q$ in $G$. Moreover, since
$S$ is of Lie type, this element will be vanishing in $G$ by Lemma \ref{defect2}.
\end{proof}

\begin{thmB} Let $G$ be a group with trivial Fitting subgroup. Then every prime divisor of \(|G|\) is in \(\Vv G\), and \(\gammava G\) is a complete graph.
\end{thmB}

\begin{proof} Let $G$ be a counterexample of minimal order to our statement; thus, $\fit  G=1$ and there exist two distinct prime divisors $p$ and $q$ of $|G|$ such that $\{p, q\}$ is not an edge of \(\gammava G\). Let
$M = S_1 \times \cdots \times S_n$ be a minimal normal subgroup of $G$ (where the $S_i$ are all isomorphic to a nonabelian simple group $S$) and let $N$ be the kernel of the action of $G$ on $\{S_1, \ldots, S_n\}$.
Also, denote $\cent G M$ by $C$ and let $\overline{G} = G/C$. We proceed through various steps.

\smallskip
\noindent (i) {\sl{$C$ is trivial.}}

Observe that $\overline{G} = G/C$ has a unique minimal normal subgroup $\overline{M} \cong M$, and therefore $\fit  {\overline{G}} =1$.
 For a proof by contradiction, let us assume $C\neq 1$.

 If $\{p,q\} \subseteq \pi(\overline{G})$ then, by the minimality of $G$, we get that $\{p,q\}$ is an edge of \(\gammava{\overline{G}}\); thus Lemma~\ref{lemma1} yields that $\{p,q\}$ is an edge of \(\gammava G\) as well, against our assumptions.

Suppose now $p,q \in \pi(C)$. As $\fit  C \leq \fit  G=1$, there exists $g \in C$ with $pq$ dividing $|g^C|$ by \cite[Theorem 9]{casolo}. 
Choose $h \in M$ with order divisible by some prime greater than $3$; then $h$ is vanishing in $G$ by Proposition \ref{prop1}. An application of Lemma \ref{products} yields that $gh$ is vanishing in $G$ with conjugacy class size divisible by $pq$, again a contradiction.

Thus, we can assume that $p$ divides $|\overline{G}|$ and $q$ does not. As $\overline{G}$ has a unique minimal normal subgroup $\overline{M}$, 
by Lemma \ref{lemma2}, there exists an element $\overline{g} \in \overline{M}$ with $p$ dividing $|\overline{g}^{\overline{G}}|$ 
and $\overline{g}$ vanishing
in $\overline{G}$. As $C$ does not have a central Sylow $q$-subgroup (because $\fit  C=1$), there exists an element
$h \in C$ with $q$ dividing $|h^C|$. For $g$ a preimage of $\overline{g}$ in $M$, by Lemma \ref{products}, we have that $gh$ is
vanishing in $G$ with conjugacy class size divisible by $pq$. This contradiction proves the claim.

Note that, for every \(i\in\{1,...,n\}\), the factor group \(N/\cent N{S_i}\) is an almost simple group with socle isomorphic to \(S\). Moreover, as \(C=1\), \(N\) can be embedded in the direct product of the factor groups \(N/\cent N{S_i}\), and since these all have the same order, we get \(\pi(N)=\pi(N/\cent N{S_i})\).

\smallskip
\noindent (ii) {\sl{$N$ is a proper subgroup of $G$.}}

In fact, if $N=G$, then $M$ is a simple group and $G$ is an almost simple group with socle $M$. The desired conclusion now follows from Proposition~\ref{almost}. 

\smallskip
\noindent (iii) {\sl{We have $\{p,q\}\not \subseteq \pi(G/N)$.}}

In fact, assuming the contrary, choose nonempty $\Gamma_1, \Gamma_2 \subseteq \Omega = \{S_1, \ldots, S_n \}$ with $\Gamma_1 \cap \Gamma_2 = \emptyset$ and $\{p,q\}\subseteq \pi(|G:G_{\Gamma_1} \cap G_{\Gamma_2}|)$ (see Proposition~\ref{three}). Also, let $u$ and $v$ be nontrivial elements of different orders in $S\simeq S_i$, such that the order of $v$ is divisible by some prime $r$ greater
than 3. Now, for $S_{\alpha} \in \Gamma_1$ and $S_{\beta} \in \Gamma_2$
let $u_{\alpha}$ and $v_{\beta} $ correspond, respectively, to \(u\) and \(v\) (via the isomorphisms \(S_{i}\simeq S\)). As in the proof of Lemma~\ref{lemma2}, we see that 
$g = \Pi u_{\alpha}\Pi v_{\beta}$ is vanishing
in $G$ and its conjugacy class size is divisible by \(pq\), contradicting the hypotheses.

\smallskip
\noindent (iv) {\sl{We have $\{p,q\}\not \subseteq \pi(N)$.}}

In view of the last paragraph of Claim (i), if $p$ and $q$ both divide $|N|$, then they both divide the order of \(\overline{N}=N/\cent N{S_1}$ as well. By Proposition \ref{almost}, there exists $\overline{g} \in\overline{S_1}$ such that
$pq\mid|\overline{g}^{\overline{N}}|$. Let $g \in S_1$ be a preimage of $\overline{g}$, and set $h \in S_2$ to be an element with order divisible
by a prime greater than 3. Then $gh$ is vanishing
in $G$ by Proposition \ref{prop1}; moreover, as $\overline{gh}=\overline{g}$, we get $pq\mid |\overline{gh}^{\overline{N}}|$. But $|\overline{gh}^{\overline{N}}|$ divides $|(gh)^N|$, which in turn divides $|(gh)^G|$, against the hypotheses. 

\smallskip
\noindent (v) {\sl{Final contradiction.}}

Our conclusion so far is that we can assume $p \in \pi(G/N)$ and $q \in \pi(N)$. Then $q$ divides $|N/\cent N{S_1}|$ and, as in the previous claim, there
exists an element $u \in S_1$ such that $q$ divides $|u^{N}|$ (note that \(u\) is a vanishing element of \(G\) unless possibly when it is a \(\{2,3\}\)-element). Now choose nonempty subsets $\Gamma_1$ and $\Gamma_2$ of
$\Omega = \{S_1, \ldots, S_n\}$ such that $p$ divides $|G: G_{\Gamma_1} \cap G_{\Gamma_2}|$. Take \(v\in S\) with \(o(v)\neq o(u)\) and such that, in the case when \(u\) is a \(\{2,3\}\)-element, the order of \(v\) is divisible by a prime greater than \(3\). As in the proof of Lemma~\ref{lemma2}, define $g = \Pi u_{\alpha} \Pi v_{\beta}$ where $u_{\alpha}$ and \(v_{\beta}\) correspond to \(u\) and \(v\) respectively. As already seen, $g$ is vanishing in $G$ and $|g^G|$ is divisible by $p$. Moreover, as we can assume \(S_1\in\Gamma_1\), the image of $g$ in $N/\cent N{S_1}$
is $u$, and therefore $q$ divides $|g^{N}|$, which in turn divides $|g^G|$. Thus $g$ is a vanishing
element of $G$ with conjugacy class size divisible by $pq$, and this contradiction completes the proof of the theorem.
\end{proof}

\section{Theorem A}

In this section we prove the main result of this paper. Our key tool will be Lemma~\ref{regular} and its consequence Corollary~\ref{coro}. We start with a lemma concerning vanishing elements of alternating groups.

\begin{lemma}
\label{goodpartitions}
For \(n\geq 7\) and \(t\geq 2\), let \(x\) be a permutation in \(\alt n\) whose type (including fixed points) is \((t,...,t)\) or \((t,...,t,1)\). Then there exists an irreducible character \(\theta\) of \(\alt n\) which has an extension to \(\sym n\), and such that \(\theta(x)=0\).
\end{lemma}

\begin{proof}
Assume that \(x\) is of type \((t,...,t)\), whence \(n=kt\) for a suitable integer \(k\). If \(k\geq 3\), consider the partition \(\mu=(n-t-1,t,1)\) of \(n\) (which is not self-associate), and let \(\chi_{\mu}\in\irr{\sym n}\) be the corresponding character. Then the Murnaghan-Nakayama formula (see for instance \cite[Theorem 4.10.2]{sagan}) yields \(\chi_{\mu}(x)=0\), and \(\theta\) can be chosen to be the (irreducible) restriction of \(\chi_{\mu}\) to \(\alt n\). In the case when \(k\leq 2\), we can argue as above using the partition \((n-3, 2,1)\).

Finally, if \(x\) is of type \((t,...,t,1)\), then we use the partition \((n-1,1)\).
\end{proof}

Let \(V\) be an \(n\)-dimensional vector space over a field \(\Ef\), and let \(D\) be the set consisting of the elements \((\alpha_1,...,\alpha_n)\in V\) such that \(\sum\alpha_i=0\). The \((n-1)\)-dimensional vector space \(D\) has the natural structure of an \(\Ef[\alt n]\)-module, where \(\alt n\) acts by permuting coordinates, and we will be interested in the case when the characteristic \(q\) of \(\Ef\) is coprime with \(|\alt n|\), i.e., \(q>n\); in this case, \(D\) turns out to be an irreducible  \(\Ef[\alt n]\)-module, which is called the \emph{deleted permutation module} over \(\Ef\) (see \cite{goodwin}).

We will deal with the module \(D\) regarded as an (irreducible) \(\Ef^{\times}\times \alt n\) module over \(\Ef\), where \(\Ef^{\times}\) denotes the multiplicative group of \(\Ef\), acting on \(D\) by scalar multiplication.

\begin{lemma}
\label{deleted} 
Let \(D\) be the deleted permutation module for \(\alt n\) over the field \(\Ef\), where the characteristic of \(\Ef\) is larger than \(n\). Let \(d=(\alpha_1,\alpha_2,...,\alpha_{n-1},\beta)\in D\) be such that the \(\alpha_i\) are pairwise distinct elements of \(\Ef^{\times}\). Regarding \(D\) as a module for \(G=\Ef^{\times}\times\alt n\), assume that the element \((\lambda,x)\in G\) centralizes \(d\). Then, if \(t\) is the order of \(x\), the type of \(x\) (including fixed points) is either \((t,...,t)\) or \((t,...,t,1)\).
\end{lemma}

\begin{proof}
Let \(d\) be an element of \(D\) as in the statement, and assume \((\lambda,x)\in \cent G d\). We first observe that, if \(\lambda=1\), then \(x\) is also \(1\). In fact, if \(x\) maps the symbol \(j\in\{1,...,n-1\}\) to \(jx\neq j\) then, as \((1,x)\in \cent G d\), the corresponding entries of \(d\) must coincide; but, the \(\alpha_i\) being pairwise distinct, we get \(jx=n\) and \(\alpha_j=\beta\). As a consequence, the only nontrivial cycle in \(x\) is possibly \((j,n)\), but this is a contradiction since \(x\) is an even permutation.

Next, let \(j\in\{1,...,n-1\}\) be a symbol lying in an \(\langle x\rangle\)-orbit of length \(s\), so \(s\) is a divisor of \(t=o(x)\). Since \((\lambda^s,x^s)\) centralizes \(d\), the \(j\)th entries of \(d\) and of \(d^{(\lambda^s,x^s)}\) must coincide; therefore we get \(\lambda^s\alpha_j=\alpha_j\), whence \(\lambda^s=1\). This yields that \((1,x^s)=(\lambda^s,x^s)\) centralizes \(d\), therefore \(x^s=1\) by the paragraph above, and \(s=t\). 

We conclude that every cycle of \(x\) involving at least one symbol in \(\{1,...,n-1\}\) has in fact length \(t\), and the proof is complete.
\end{proof}

\begin{lemma} 
\label{regular}
Let \(S\) be a nonabelian simple group that is not of Lie type, and let \(M=S_1\times\cdots\times S_k\) be the direct product of \(k\) copies of \(S\). Let \(q\) be a prime not dividing \(|S|\), and \(A\) a faithful \(M\)-module over the field \(\F\) with \(q\) elements. Assume that there are no regular orbits for the action of \(M\) on \(A\). Then \(S\) is isomorphic to \(\alt n\) for some \(n\geq 7\); moreover, there exists \(a\in A\) such that, for \(x=s_1\cdots s_k \in \cent M a\), each \(s_i\) is a permutation of type \((o(s_i),...,o(s_i))\) or \((o(s_i),...,o(s_i), 1)\).
\end{lemma}

\begin{proof} Since \(q\) is coprime with the order of \(M\), the \(\F[M]\)-module \(A\) is semisimple. Assuming that there is no regular orbit for the action of \(M\) on \(A\), our aim will be to construct an element \(a\) of \(A\) yielding the desired conclusions; to this end, we will choose suitable vectors from each simple constituent of \(A\).  

Let \(V\) be such a constituent. Up to renumbering, we can assume that the kernel of the action of \(M\) on \(V\) is either trivial (and in this case we set \(h=k\)) or \(S_{h+1}\times\cdots\times S_k\) for a suitable \(h\in\{1,...,k-1\}\), so \(N=S_1\times\cdots\times S_h\) acts faithfully on \(V\). Let \(\Ef\) be a finite field extension of \(\F\) that is a splitting field for \(S\), and consider the faithful \(\Ef[N]\)-module \(V^{\Ef}=V\otimes_{\F}\Ef\); if \(W\) is an irreducible constituent of \(V^{\Ef}\) and \(\chi\) is the corresponding character, we denote by \(\K\) the field extension of \(\F\) obtained by adjoining the set of values \(\{\chi(n)\mid n\in N\}\) to \(\F\) (so, \(\K\) is a subfield of \(\Ef\)). By Theorem~1.16 in \cite[VII]{huppert}, that we will freely use with no further reference throughout this proof, we get \[V^{\Ef}\simeq \bigoplus_{\xi\in{\rm{Gal}}(\K|\F)} W^{\xi}.\] Observe that, for \(w\in W\), the vector \[v=\sum_{\xi\in{\rm{Gal}}(\K|\F)} w^{\xi}\] lies in \(V\), and we have \(\cent N v\leq \cent Nw\) (see \cite[Lemma~3.1]{fleishmann}).
As \(W\) is a simple \(S_1\times\cdots\times S_h\)-module over a splitting field for each of the \(S_i\), Theorem~3.7.1 in \cite{gor} yields that the module \(W\) decomposes as a tensor product \(W_1\otimes\cdots\otimes W_h\), where each \(W_i\) is a simple \(\Ef[S_i]\)-module. In this situation, given an element of \(W\setminus\{0\}\) of the form \(w=w_1\otimes\cdots\otimes w_h\), it can be checked that \(x=s_1\cdots s_h\in N\) centralizes \(w\) if and only if \(w_i^{s_i}=\lambda_i^{-1}w_i\) for all \(i\in\{1,...,h\}\), where the \(\lambda_i\) are in \(\Ef\) and \(\prod \lambda_i=1\); in other words, if \(x\) centralizes \(w\) then, regarding \(W_i\) as an \(\Ef^{\times}\times S_i\)-module (\(\Ef^{\times}\) acting by scalar multiplication) the element \((\lambda_i, s_i)\) centralizes \(w_i\) for every \(i\in\{1,...,h\}\). 

Now, assume that \(\Ef^{\times}\times S_i\) does not have any regular orbit on \(W_i\), and let \(\K_i\) be the field extension of \(\F\) obtained by adjoining the values of the character of \(W_i\) (as a module for \(S_i\)). Denoting by \(Z\) an irreducible constituent of \(W_i\) regarded as a \(\K_i[S_i]\)-module, we get \(W_i\simeq Z^{\Ef}\), and so the field of values of the \(\K_i[S_i]\)-module \(Z\) is \(\K_i\) as well. We claim that the group \(\K_i^{\times}\times S_i\) does not have any regular orbit on \(Z\): assuming, for a proof by contradiction, that \(z\) lies in such an orbit, it is easy to check that the vector \(z\otimes 1\in W_i\) lies in a regular orbit for the action of \(\Ef^{\times}\times S_i\), which is not possible. Now, if \(Z_0\) denotes the \(\K_i[S_i]\)-module \(Z\) viewed as an \(\F[S_i]\)-module, \(Z_0\) turns out to be irreducible, and \(\K_i^{\times}\times S_i\) does not have any regular orbit on it: we are in a position to apply a theorem by D. P. M. Goodwin (\cite[Theorem~1]{goodwin}; see also \cite[Theorem~2.1 and Theorem~2.2]{koehler}), getting that \(S_i\simeq\alt n\) for some \(n\geq 7\) and \(Z_0\) is the deleted permutation module for \(S_i\) over \(\F\). Since this module is absolutely irreducible, we get \(\K_i=\F\), and \(W_i\simeq Z_0^{\Ef}\) is the deleted permutation module for \(S_i\) over \(\Ef\). Whenever we are in this situation (which does occur for some constituent \(V\), as otherwise \(M\) would have regular orbits on \(A\) against our assumptions), we choose \(d_i\in W_i\) as in Lemma~\ref{deleted}; this can be applied because, \(q\) being coprime with \(S\), we certainly have \(q>n\). In all other cases, choose \(d_i\) lying in a regular orbit for the action of \(\Ef^{\times}\times S_i\) on \(W_i\). Then set \(w=d_1\otimes\cdots\otimes d_h\in W\), and finally define \(v=\sum_{\xi\in{\rm{Gal}}(\K|\F)} w^{\xi}\in V\).

To sum up, let \(A=V_1\oplus\cdots\oplus V_n\) be a decomposition of \(A\) into simple \(\F[M]\)-constituents. For each \(V_j\), let \(v_j\) be a vector as defined in the paragraph above, and let \(a=\sum v_j\). In view of Lemma~\ref{deleted}, it can be checked that such an element \(a\) satisfies the conclusions of our statement. 
\end{proof}

The following consequence of Lemma~\ref{regular} is helpful in locating vanishing elements.

\begin{corollary}
\label{coro}
Let \(G\) be a group, and \(A\) an abelian minimal normal subgroup of~\(G\). Let \(M/N\) be a chief factor of \(G\) such that \(|A|\) is coprime with \(|M/N|\), and \(N=\cent M A\). Then every element of \(M\setminus N\) is a vanishing element of \(G\).
\end{corollary}

\begin{proof} Set \({\overline{G}}=G/N\) and, adopting the bar convention, assume that there exists an element \(a\) of \(A\) lying in a regular orbit for the action of \({\overline M}\); then, by coprimality, the same happens for the action of \({\overline M}\) on \(\irr A\). In other words, there exists an irreducible character \(\theta\) of \(A\) such that \(I_G(\theta)\cap M=N\), and every element \(x\) of \(M\setminus N\) clearly does not lie in \(\bigcup_{g\in G}I_G(\theta^g)\). Now, if \(\psi\) is an irreducible character of \(I_G(\theta)\) lying over \(\theta\), then \(\psi^G\) is an irreducible character of \(G\) which vanishes on \(G\setminus\bigcup_{g\in G} I_G(\theta^g)\), thus in particular \(\psi^G(x)=0\), as required.

Therefore, we may assume that there are no regular orbits for the action of \({\overline M}\) on \(A\). Since, by a well-known consequence of Brodkey's Theorem (\cite{brodkey}), regular orbits do exist if \({\overline M}\) is abelian, we may focus on the case when \({\overline M}= S_1\times\cdots\times S_k\), where the \(S_i\) are pairwise isomorphic nonabelian simple groups. 

Observe that every nontrivial element of \({\overline M}\) is a vanishing element of \({\overline{G}}\) provided \(S_1\) is a group of Lie type (Theorem~\ref{defect} and Lemma~\ref{defect2}), so we easily get the desired conclusion in this case. Therefore we can assume that \(S_1\) is not of Lie type, and we can apply Lemma~\ref{regular} with respect to the action of \({\overline{M}}\) on \(A\): we get \(S_1\simeq\alt n\) for some \(n\geq 7\), and we can choose an element \(a\in A\) satisfying the conclusions of Lemma~\ref{regular}. Now, by coprimality, we can consider a character \(\theta\in{\rm{Irr}}(A)\) whose inertia subgroup \(I_{\overline{M}}(\theta)\) coincides with \(\cent{\overline{M}} a\). 

Let \(x\) be an element in \(M\setminus N\). If \({\overline{x}}\) lies in \(I_{\overline{M}}(\theta)\), then each factor of \({\overline{x}}\) (in its decomposition into a product of elements of the \(S_i\)) is a permutation with the cyclic structure prescribed by Lemma~\ref{regular}, and the same of course holds also if \({\overline{x}}\) centralizes a \(G\)-conjugate of \(\theta\). In other words, either \({\overline{x}}\) does not centralize any \(G\)-conjugate of \(\theta\) (and in that case \(x\) is vanishing in \(G\), by the argument in the first paragraph of this proof), or it is as in the conclusions of Lemma~\ref{regular}. In the latter case, an application of Lemma~\ref{goodpartitions} (together with Lemma~5 of \cite{bianchi}) yields that \({\overline{x}}\) is a vanishing element of \({\overline{G}}\), and the proof is complete.
\end{proof}

We are ready to prove Theorem A, that we state again.

\begin{thmA} Let \(G\) be a group, and suppose \(G\) has a nonabelian minimal normal subgroup. If \(p\) and \(q\) are in \(\V G\), but there is no vanishing conjugacy class of \(G\) whose size is divisible by \(pq\), then \(G\) is \(\{p,q\}\)-solvable.
\end{thmA}

\begin{proof}
Let \(G\) be a counterexample to the statement, having the smallest possible order. Since \(p\) and \(q\) are nonadjacent vertices of \(\gammava G\), Theorem~B yields \(\fit G\neq 1\). As a consequence, there exists an abelian minimal normal subgroup \(A\) of \(G\). We will  proceed through a number of steps.

\smallskip
\noindent (i) {\sl{$A$ is the unique abelian minimal normal subgroup of \(G\). Moreover, we can assume \(p\in \V{G/A}\) and $q \not\in \V{G/A}$.}}

The factor group \(G/A\) clearly has a nonabelian minimal normal subgroup. If \(p\) and \(q\) are vertices of \(\Gamma(G/A)\), then they cannot be adjacent in \(\gammava{G/A}\), as otherwise they would be adjacent in \(\gammava G\) as well; therefore, by our minimality assumption, \(G/A\) is \(\{p,q\}\)-solvable. But then \(G\) is so, and this is a contradiction.   On the other hand, if both \(p\) and \(q\) are not in \(\V{G/A}\), then \(G/A\) is both \(p\)-nilpotent and \(q\)-nilpotent and \(G\) is not a counterexample. Hence \(\{p,q\}\cap \V{G/A}\) cannot be empty, and we will assume $p \in \V{G/A}$,  $q \not\in \V{G/A}$ (thus \(G\) is \(q\)-solvable). Now, let $B\neq A$ be an abelian minimal normal subgroup of \(G\). The above discussion applies to \(B\) as well, hence \(\{p,q\}\cap \V{G/B}\) contains precisely one element; if this element is \(q\), then $G/B$ is \(p\)-solvable and so is \(G\), a contradiction. But if $q\not\in \V{G/B}$, then both \(G/A\) and \(G/B\) have a central Sylow \(q\)-subgroup, and the same holds for \(G\), which embeds into \((G/A)\times(G/B)\). This would imply \(q\not\in\V G\), which is not the case. Therefore, assuming the existence of a minimal normal subgroup of \(G\) other than \(A\) we get a contradiction, and the claim is proved.

\smallskip
\noindent (ii) {\sl The Frattini subgroup $\frat G$ of \(G\) is trivial, and $\fit G=A$.}

Assume $\frat G \neq 1$, so $A \leq \frat G$. Let $Q$ be a Sylow $q$-subgroup of $G$. Then, as we are assuming \(q\not\in \V{G/A}\), we have $QA \trianglelefteq G$ and hence $Q \unlhd G$ by the Frattini argument. As \(A\) is the unique abelian minimal normal subgroup of \(G\), this clearly implies that \(A\) is a \(q\)-group. Moreover, \(G/A\) has a normal \(q\)-complement \(K/A\) and, again by the Frattini argument, a Hall \(q'\)-subgroup \(K_0\) of \(K\) is actually a normal \(q\)-complement of \(G\), thus \(G=K_0\times Q\). Observe that \(Q\) cannot be abelian, as otherwise it would lie in \(\zent G\), thus we can choose \(x\in Q\setminus\zent Q\); note that \(|x^G|\) is divisible by \(q\), and that \(x\) is a vanishing element of \(Q\) (see~\cite[Theorem B]{INW}), whence clearly a vanishing element of \(G\) as well. Consider now \(y\in K_0\) such that \(|y^G|\) is divisible by \(p\). Then \(xy\) is a vanishing element of \(G\) whose conjugacy class in \(G\) has size divisible by \(pq\), and \(G\) is not a counterexample. We conclude that \(\frat G=1\); in particular, \(\fit G\) is a direct product of abelian minimal normal subgroups of \(G\) (see \cite[III.4.5]{Huppert1}), whence \(\fit G=A\).

\smallskip
\noindent (iii) {\sl{\(A\) is a \(q\)-group}}

For a proof by contradiction, we assume that \(q\nmid |A|\). Hence, if \(Q\) is a Sylow \(q\)-subgroup of \(G\), we get $A = [A,Q] \times \cent AQ$. Note that \([A,Q]\) cannot be trivial, as otherwise the abelian Sylow subgroup \(Q\) of \(G\) would be normal in \(QA\nor G\) and hence central in \(G\). Moreover, every \(G\)-conjugate of \(Q\) is of the form \(Q^a\) for a suitable \(a\in A\), whence $[A,Q]$ is normal in \(G\). As a consequence, we get \([A,Q]=A\) and \(\cent A Q=1\); in particular, as this holds for any Sylow \(q\)-subgroup of \(G\), for every \(x\in A\setminus\{1\}\) we have that \(q\) is a divisor of \(|x^G|\). 

Now, let \(M\) be a nonbelian minimal normal subgroup of \(G\), and set \(C=\cent G M\). If $p$ divides $|G/C|$ then, by Lemma~\ref{lemma2}, there exists \(y\in M\) which is a vanishing element of \(G\) such that \(p\mid|y^G|\). But then, for any \(x\in A\setminus\{1\}\), \(xy\) is a vanishing element of \(G\) with \(pq\mid|(xy)^G|\), a contradiction. We conclude that \(C\) contains a Sylow \(p\)-subgroup \(P\) of \(G\), and clearly \(P\not\leq \zent C\) (as otherwise \(P\) would be normal in \(G\), so \(G\) would be \(p\)-solvable). Moreover, any Sylow \(q\)-subgroup \(Q\) of \(G\) lies in \(C\) because \([Q,G]\leq A\), and therefore \([Q,M]\leq A\cap M=1\); but \(Q\) is not centralized by \(A\), so \(Q\not\leq \zent C\) as well. As a consequence, \(p\) and \(q\) are vertices of \(\Gamma(C)\). Observe that \(\{p,q\}\) must be an edge of \(\Gamma(C)\), as otherwise \(C\) would be \(\{p,q\}\)-solvable by \cite[Theorem~B(i)]{casolo}, and so would be \(G\), because \(G/C\) is a \(\{p,q\}'\)-group. Hence, there exists an element \(c\in C\) such that \(|c^G|\) is divisible by \(pq\). Take any \(y\in M\) which is vanishing in \(G\): we get that \(xy\) is vanishing in \(G\) and \(|(xy)^G|\) is divisible by \(pq\).
  
\smallskip
\noindent (iv) {\sl{Final contradiction}}.

By the previous step, we deduce that \(A\) is a Sylow \(q\)-subgroup of \(G\); in fact, if \(Q\in\syl q G\), then \(Q/A\) is normal in \(G/A\), and so \(Q\leq\fit G=A\). In particular, for every \(x\in G\setminus\cent G A\), we have \(q\mid|x^G|\).
Observe also that \({\overline{G}}=G/\cent GA\) does not have any vanishing element whose \(\overline{G}\)-class size is divisible by \(p\) (otherwise, if \(\overline{x}\) is such an element, then \(x\) would be a vanishing element of \(G\) such that \(pq\mid|x^G|\)); whence, by Theorem~A in \cite{dolfi}, \({\overline{G}}\) is \(p\)-nilpotent. 

Let now \({\overline{K}}=K/\cent G A\) be a minimal normal subgroup of \({\overline{G}}\). Clearly \(K\) does not have a central Sylow \(q\)-subgroup, because \(K>\cent G A\); but \(K\) does not have a central Sylow \(p\)-subgroup as well, since otherwise \(K\) would be \(p\)-solvable and so would be \(G\) (recall that \({\overline{G}}\) is \(p\)-nilpotent). As a consequence, \(p\) and \(q\) are vertices of \(\Gamma(K)\). Moreover, \(K\) must have conjugacy classes of size divisible by \(pq\); otherwise \(K\) would be \(p\)-solvable by \cite[Theorem~B(i)]{casolo} and, as above, \(G\) would be \(p\)-solvable. Note that such conjugacy classes of \(K\) are contained in \(K\setminus\cent G A\). But an application of Corollary~\ref{coro} yields that every element in \(K\setminus\cent G A\) is actually a vanishing element of \(G\), and this is the final contradiction that completes the proof.
\end{proof}

As for the Corollary stated in the Introduction, observe that if \(G\) is a group having a nonabelian minimal normal subgroup, and \(p\) is a prime that is not a complete vertex of \(\gammava G\), then either \(p\not\in\Vv G=\V G\) (and so \(G\) is \(p\)-nilpotent) or \(p\in\Vv G\) is not adjacent in \(\gammava G\) to some other vertex \(q\). In the latter case, Theorem~A  yields the \(p\)-solvability of \(G\).

\section*{Acknowledgements} The fourth author wishes to thank Silvio Dolfi for a couple of very pleasant discussions on the subject of this paper.

\end{document}